\documentclass[11pt]{amsart}
\usepackage{amsfonts}
\usepackage{hyperref}
\usepackage[utf8]{inputenc}
\usepackage[T1]{fontenc}
\usepackage[dvips]{graphicx}
\usepackage{amssymb}
\usepackage{amsmath}
\usepackage{dsfont}
\usepackage{color}
\usepackage{latexsym}
\usepackage{bbm}
\usepackage{color}
\usepackage{amsthm}
\usepackage{multicol}
\usepackage{tikz}
\usepackage[top   = 2.75cm,
bottom = 2.50cm,
left   = 2.50cm,
right  = 2.00cm]{geometry}
\usetikzlibrary{calc,decorations.pathreplacing}
\bibstyle{plain}
\theoremstyle{plain}
\newtheorem{theorem}{Theorem}
\newtheorem{proposition}[theorem]{Proposition}
\newtheorem{lemma}[theorem]{Lemma}
\newtheorem*{BT}{Banakiewicz Theorem}
\newtheorem*{obs}{Observation}
\newtheorem{problem}[theorem]{Problem}
\newcommand{\N}{\mathbb N}

\begin{document}

\title[Uniqueness of atomic measures]{A characterization of uniqueness of purely atomic finite measures with central Cantor set range}

\newcommand{\eacr}{\newline\indent}

\author{Piotr Nowakowski}
\address{Faculty of Mathematics and Computer Science
\\University of Lodz
\\Banacha 22,
90-238 \L\'{o}d\'{z}
\\Poland\\
%\\{\textcolor{orcid-green}{\aiOrcid}}
 ORCID: 0000-0002-3655-4991}
\email{piotr.nowakowski@wmii.uni.lodz.pl}

\author{Franciszek Prus-Wi\'{s}niowski}
\address{\llap{*\,}Franciszek Prus-Wi\'{s}niowski\eacr
Instytut Matematyki\eacr
Uniwersytet Szczeci\'{n}ski\eacr
ul. Wielkopolska 15\eacr
PL-70-453 Szczecin\eacr
Poland\acr ORCID 0000-0002-0275-6122}
\email{franciszek.prus-wisniowski@usz.edu.pl}


%\date{\today}

\newcommand{\acr}{\newline\indent}

\begin{abstract}
We study purely atomic measures whose range is a central Cantor set and characterize those central Cantor sets that are the range of exactly one such measure. Next, we extend the characterization to the case of symmetric Cantor sets using a different method of proof. Finally, we make a few initial observations and remarks on recovering a measure whose range is a Cantorval.
\end{abstract}
\maketitle
Let $\mu$ be a finite atomic measure with infinite range. We can identify it with a measure $\mu$ on $\mathbb N$ such that $\mu({n})\ge\mu({n+1})>0$ for all $n$. Denoting $x_n:=\mu({n})$, we get $\sum x_n<+\infty$, since $\mu$ is finite. Thus, the range of $\mu$ is
$$
\textnormal{rng}(\mu)\ = \left\{\,\mu E:\quad E\subset \mathbb N\,\right\}\ =\ \left\{y\in\mathbb R:\quad \exists A\subset \mathbb N\quad y=\sum_{n\in A}x_n\,\right\}.
$$
The latter set is the achievement set of the sequence $(x_n)$ \cite{Jones}, \cite{BFPW1}. Since our note addresses the problem of recovering a finite measure with infinite range, we will consider achievement sets only of summable sequences of positive, non-increasing terms. Henceforth, all sequences discussed in this note are assumed to be summable, positive, and arranged in a non-increasing manner, unless explicitly stated otherwise. We will say that an achievement set $E(z_n)$ is uniquely achievable (or uniquely determined) if there is no other monotone sequence $(y_n)$ such that $E(z_n)=E(y_n)$. Thus, according to the Guthrie-Nymann Classification Theorem \cite[Theorem 1]{GN88} (see also \cite[Theorem 1]{NS}), the range of $\mu$ is either a Cantor set or a Cantorval. A Cantorval is a subset of the reals that is bounded, regularly closed and whose boundary is a Cantor set. More geometrically,  a Cantorval is any set subset of the reals homeomorphic to the set
$$
GN\ :=\ C\,\cup\,\bigcup_n \mathcal{G}_{2n-1}
$$
where $C$ is the Cantor ternary set $C=E(\tfrac2{3^n})$ and $\mathcal{G}_{2n-1}$ is the union of all $4^{n-1}$ $C$-gaps removed from $[0,1]$ in the ($2n-1$)-th step of the standard geometric construction of $C$ \cite{GN88}, \cite{MNP}.

 Not long ago, Bartoszewicz, G\l\c{a}b and Marchwicki investigated the problem of recovering a purely atomic finite measure from its range \cite{BGM18}, aiming to answer a question of T. Banakh as to whether Cantor achievement sets are uniquely defined (that is, whether they are achievement sets of only one sequence). They found three different sufficient conditions for $E(x_n)$ to be a uniquely achievable Cantor set \cite[Theorems 4.2, 4.7, 4.8]{BGM18}. The first of these says that if $x_n>2x_{n+1}$ for all $n$, then the equality $E(x_n)=E(y_n)$ implies that $(x_n)=(y_n)$. The conditions $x_n>2x_{n+1}$ for all $n$ imply that $x_n>r_n:=\sum_{k>n}x_k$ for all $n$, that is, the sequence $(x_n)$ is fast convergent. Fast convergence is a very natural sufficient condition for $E(x_n)$ to be a Cantor set. It was already discovered by the father of achievement set theory, Soichi Kakeya \cite{Kakeya}, \cite{Kakeya2}.

Given a summable sequence $(x_n)$ of positive, non-increasing terms, we have
$$
E\ =\ E(x_n)\ \ =\ \ E_1^k(x_n)\ +\ E_k(x_n)
$$
for every $k\in\mathbb N$, where
$$
E_1^k\ =\ E_1^k(x_n)\ :=\ \left\{\sum_{n\in A}x_n:\quad A\subset\{1,\ldots,k\}\,\right\}
$$
is the set of \textsl{$k$-initial subsums} of $(x_n)$ and
$$
E_k\ =\ E_k(x_n)\ :=\ \left\{\sum_{n\in A}x_n:\quad A\subset\{k+1,k+2, \ldots\,\}\,\right\}
$$
is the achievement set of the $k$-th remainder of $(x_n)$, that is, of the sequence $(x_n)_{n>k}$. The set of $k$-initial subsums is a special case of a more general definition that we will need in the proof of Lemma \ref{l6} :
$$
E_m^k\ =\ E_m^k(x_n)\ :=\ \left\{\sum_{n\in A}x_n:\quad A\subset\{m, m+1,\ldots,k\}\,\right\}
$$
where $m,k$ are positive integers such that $m\le k$.

The $k$-th iterate of $E$ is defined by
$$
I_k\ :=\ E_1^k\,+\,[0, r_k]\ \ =\ \ \bigcup_{f\in E_1^k}[f,f+r_k].
$$
Each $I_k$ is a multi-interval set. We also write $I_0:=[0,r_0]$ and call this interval \textsl{the fundamental interval} of $E$. Every achievement set is symmetric with respect to the center of its fundamental interval.

Always $E=\bigcap_kI_k$ \cite[Fact 21.8]{BFPW1}. The family of all connectivity components of $I_{k-1}\setminus I_k$ is denoted by $\mathcal{G}_k$, and its elements are called $E$-gaps of order $k$. It is not difficult to see that the family $\mathcal{G}_k$ is nonempty if and only if $r_k<x_k$. The First Gap Lemma \cite[Fact 2.19]{BFPW1} says that if $r_k<x_k$, then the interval $(r_k,x_k)$ is a gap of $E$. Gaps of this form are called \textsl{principal}. We will say that an $E$-gap $G$ is \textsl{dominating} if all $E$-gaps lying to the left of $G$ are shorter than $G$. The Third Gap Lemma \cite[Lemma 4]{BFGPWS} says that every dominating gap is principal.

We should keep in mind that being a gap of order $k$ may depend on a particular representation of an achievable set $E$. Consider two sequences $(x_n)$ and $(z_n)$ given by $x_1=2$, $x_2=1$, $x_n=\frac2{3^{n-2}}$ for $n\ge3$ and $z_1=z_2=z_3=1$, $z_n=\frac2{3^{n-3}}$ for $n\ge4$. Then $E:=E(x_n)=E(z_n)$ is a Cantor set, and $(\frac13,\frac23)$ is an $E$-gap of order 2 with respect to the representation $E=E(x_n)$ and an $E$-gap of order 3 with respect to the representation $E=E(z_n)$. On the other hand, being a dominating gap does not depend on the representation. In fact, the concept of domination of a gap makes sense for any compact subset of the real line.

Sometimes we discuss two sequences at once, say, $x=(x_i)$ and $y=(y_i)$. In order to distinguish their remainders, we will use an additional superscript: $r^x_n:=\sum_{i>n}x_i$ and $r^y_n:=\sum_{i>n}y_i$.

A sequence $(y_i)$ of positive, non-increasing terms converging to 0 is said to be \textsl{semi-fast convergent} if
$$
y_n\ >\ \sum_{i:,y_i<y_n}y_i \qquad\text{for all $n\in\mathbb N$.}
$$
In particular, $\sum_iy_i<+\infty$. A strictly decreasing sequence $(y_i)$ is semi-fast convergent if and only if it is fast convergent. If $(\alpha_p)_{p\in\mathbb N}$ is the decreasing sequence of all values of a semi-fast convergent sequence $(y_i)$, then there is a unique sequence $(M_p)_{p\in\mathbb N}$ of positive integers such that
$$
y_i\ =\ \alpha_p \quad\text{for}\ \ \sum_{k=0}^{p-1}M_k\ <\ i\ \le\ \sum_{k=0}^pM_k
$$
where $M_0:=0$. The numbers $M_p$ are the multiplicities of the corresponding values $\alpha_p$ in the sequence $(y_i)$. It is customary to identify $(y_i)=(\alpha_p;M_p)$. Then the sum of the sequence is $\sum_iy_i\ =\ \sum_p(\alpha_p;M_p)\ =\ \sum_{p=1}^\infty \alpha_pM_p$. The symbol $E(\alpha_p;M_p)$ stands for the achievement set $E(y_i)$. Define $R_p:=\sum_{i>p}\alpha_iM_i$. The achievement set of a semi-fast convergent sequence $(y_i)=(\alpha_p;M_p)$ has gaps of order $k$ if and only if $k=\sum_{j=1}^pM_j$ for some $p$. Then each gap of such an order $k$ has length $\alpha_p-R_p$. Achievement sets of semi-fast convergent series are always Cantor sets \cite[Theorem 16]{BFPW2}.

The most important family of sequences from the point of view of achievement set theory is the family of multigeometric sequences, because they are relatively easy to handle computationally. A \textsl{multigeometric sequence} $(k_1,\ldots,k_m;,q)$ is a mixture of $m$ positive geometric sequences with the same ratio $q\in(0,1)$, that is, a sequence of the form $(k_1q,k_2q,\ldots, k_mq,k_1q^2,\ldots,k_mq^2,k_1q^3,\ldots)$. It is customary to arrange the main coefficients in non-increasing order: $k_1\ge k_2\ge\ldots\ge k_m$. There are a number of strong results concerning how the topological type of the achievement set of a multigeometric sequence depends on the main coefficients and on the ratio $q$ \cite{BFS},\cite{BBFS}. The first counterexample to the 1914 Kakeya hypothesis was provided by a multigeometric Cantorval \cite{WS}. The best-researched achievable Cantorval is the Guthrie-Nymann Cantorval $E(3,2;\frac14)$ \cite{GN88}, \cite{BPW},\cite{GM23}.

\section{ The central Cantor set range}

Recall the well-known correspondence between central Cantor sets (with left endpoint at 0) and fast convergent sequences. A \textsl{central Cantor set} is obtained by the following construction. At the initial stage, we remove from a closed interval $[0,y_0]$ (where $x_0>0$) a middle open interval $(y_1,x_1)$, just as in the standard geometric construction of the Cantor ternary set, where we remove the interval $(\frac13,\frac23)$ from $[0,1]$. The only requirement is that $0<x_1-y_1<y_0$. The remaining set, which is the union of two closed intervals of the same length, is denoted by $I_1$. The length of each of the two components of $I_1$ is $y_1$. In the second step of the construction, we remove a middle open interval from each of the components of $I_1$ in such a way that all the removed intervals have the same length, smaller than $y_1$. Denote the leftmost of these removed open intervals by $(y_2,x_2)$. In the standard geometric construction of the Cantor ternary set, it would be $(y_2,x_2)=(\frac19,\frac29)$. Proceeding inductively, in the $n$-th step of the construction, from each of the $2^{n-1}$ component intervals of $I_{n-1}$ we remove the middle open interval of length $x_n-y_n$, where $0<x_n-y_n<y_{n-1}$. $y_{n-1}$ is the length of each component of $I_{n-1}$, and $(y_n,x_n)$ is the open interval removed from the leftmost component of $I_{n-1}$. The set $\bigcap_{n\in\mathbb N} I_n$, where the iterations $I_n$ are constructed according to the procedure just described, is called a \textsl{central Cantor set}. Clearly, it is a Cantor set. The adjective central aptly reflects its nature, since all open intervals removed at all stages of its construction are concentric with the appropriate components of the preceding iterations (that is, they are central with respect to the appropriate components). Moreover, each central Cantor set is the achievement set of the series $\sum x_n$, where $x_n$'s are the right endpoints of the leftmost open intervals removed in the $n$-th step of the construction. The series $\sum x_n$ is fast convergent, and its remainders $r_n$ are exactly the left endpoints of the leftmost open intervals removed in the $n$-th step of the construction, that is, $r_n=y_n$ for all $n$.

Conversely, for any fast convergent series $\sum x_n$, its achievement set $E(x_n)$ is a central Cantor set. In particular, every central Cantor set (with left endpoint at 0) has a unique fast convergent representation $(x_n)$ described above. However, not every central Cantor set is uniquely achievable (see \cite[Example 4.6]{BBFP}).

We are now going to give a very simple characterization of the central Cantor sets that are uniquely determined. The crucial tool in the proof will be an inconspicuous but powerful theorem concerning centers of distances of achievement sets of fast convergent sequences \cite[Theorem 2]{B23}. The center of distances is a metric invariant introduced by Bielas, Plewik and Walczyńska in \cite{BPW} as a tool for detecting non-achievability. Given a metric space $(X,d)$, the center of distances of a non-empty set $S\subset X$ is
$$
\mathcal{C}(S)\ :=\ \bigl\{\alpha\ge0:\quad \forall x\in S\ \exists\ y\in S\quad d(x,y)=\alpha,\bigr\}.
$$
Clearly, $0\in\mathcal{C}(S)\subset[0,+\infty)$. Given a point $x\in X$ and a subset $A\subset X$, we define the set of distances from $x$ to the points of $A$ by
$
D_x(A):=\{d(x,y):\ y\in A\}.
$
Then the center of distances of $A$ is given by
$
\mathcal{C}(A)=\bigcap_{x\in A}D_x(A).
$
For example, the center of distances of the Sierpiński carpet $SC\subset\mathbb R^2$ is $\mathcal{C}(SC)=[0,\frac56]$. To see this, let us first prove the following simple observation. The symbol $\textnormal{Fr}\,A$ stands for the boundary of a set $A$.

\begin{proposition}
\label{pC}
Let $A$ be a compact, path-connected subset of a Euclidean space $(\mathbb R^n,d)$. Then
$$
\mathcal{C}(A)\ =\ [0,\alpha] \qquad \text{where}\qquad \alpha\ =\ \min_{x\in A} \left(\max_{y\in\textnormal{Fr}\,A}d(x,y)\right).
$$
\end{proposition}
\begin{proof})
Let $\gamma$ be an arc in $A$ connecting two points $x,y\in A$. Since the function $\gamma\ni p\mapsto d(x,p)\in[0,+\infty)$ is continuous, the
Darboux property gives $[0,d(x,y)]\subset D_x(A)$. Since we deal with a Euclidean metric, $\max_{y\in A}d(x,y)= \max_{x\in\textnormal{Fr}\,
A}d(x,y)$ and hence $D_x(A)=[0, \max_{x\in\textnormal{Fr}\,A}d(x,y)]$ for every $x\in A$. Therefore, $\mathcal{C}(A)=\bigcap_{x\in A}D_x(A)=[0,\alpha]$.
\end{proof}

In the Sierpi\'{n}ski carpet $SC$, there are exactly four points $x$ at which the minimum in Proposition \ref{pC} is attained. One of them is $(\frac23,\frac12)$, and hence
$\alpha=d\bigl((\frac23,\frac12),(0,0)\bigr)=\frac56$. Thus, $\mathcal{C}(SC)=[0,\frac56]$, by Proposition \ref{pC}, which answers a question of Bi\'{s} \cite[p.20]{KN24}. Incidentally, this also illustrates a weakness of the concept of the center of distances in the multidimensional case, since we cannot deduce the non-achievability of the set $SC$. This provides one more argument in favor of the concept of a spectre of a set, since the non-achievability of $SC$ follows easily from the shape of its spectre (see \cite[Ex. 10]{KN24}).

The concept of the center of distances was studied further in \cite{BFHLP} and \cite{K25}. Centers of distances of sets of P-sums were investigated in \cite{GM}. Centers of distances in the case of ultrametric spaces were discussed in \cite{DR26} and \cite{DRarXiv}. The usefulness of centers of distances in recognizing the non-achievability of a set follows from the following observation \cite[Prop. 3.1]{BPW}.

\begin{proposition}
\label{p1}
For any sequence $(x_n)$ of positive terms,
$$
\{0\}\cup\{x_n:\ n\in\mathbb N\,\}\ \subset\ \mathcal{C}(E(x_n)).
$$
\end{proposition}

In fact, equality may occur in Proposition \ref{p1}, as happens, for example, for geometric sequences $(q^n)$ with $q\in(0,\frac12)$ \cite[Theorem 3.3]{BPW}. These geometric sequences are fast convergent and, in light of Proposition \ref{p1}, have minimal possible centers of distances. On the other hand, there are fast convergent sequences that do not have minimal possible centers of distances \cite[Example 2.7]{BBFP}. The problem of minimality of the center of distances for achievement sets of fast convergent sequences, studied in \cite{BBFP}, was eventually solved by Banakiewicz \cite[Theorem 2]{B23}.

\begin{BT}
Let $E=E(x_n)$ be a central Cantor set given by a fast convergent sequence $(x_n)$. The center of distances of $E$ is not minimal if and only if there exists a positive integer $n$ such that one of the following possibilities holds:
\begin{itemize}
\item[(i)]\ $x_{n-2}=4x_n$ and $ x_{n-1}=2x_n$;
\item[(ii)]\ $x_{n-3}=9x_n$, $x_{n-2}=5x_n$ and $x_{n-1}=2x_n$;
\item[(iii)]\ $x_{n-3}=10x_n$, $x_{n-2}=6x_n$ and $x_{n-1}=2x_n$.
\end{itemize}
\end{BT}

In the course of the proof of our result, we will use one more observation \cite[Lemma 2.5]{BGM18}.

\begin{proposition}
\label{p2}
If $x_p=x_{p+1}=\ldots=x_{p+2j-2}$ for some $p,j$, then $jx_p\in\mathcal{C}(E(x_n))$.
\end{proposition}

We are now ready to state the main result of this note.

\begin{theorem}
\label{t4}
Let $E$ be a central Cantor set with $\textnormal{min}\,E=0$ and let $(x_n)$ be the unique fast convergent representation of $E$. Then $E$ is uniquely achievable if and only if $x_n\ne 2x_{n+1}$ for all $n\in\mathbb N$.
\end{theorem}

\begin{proof}
We will prove the following equivalent formulation of the statement: $E$ is not uniquely representable if and only if $x_n=2x_{n+1}$ for some $n\in\mathbb N$.

If $n$ is an index such that $x_n=2x_{n+1}$, then, defining
$$
y_k\ :=\ \begin{cases} x_k\qquad&\text{for $k<n$},\\
x_{n+1} &\text{for $k=n,n+1, n+2$},\\
x_{k-1} &\text{for $k\ge n+3$},
\end{cases}
$$
we obtain $(y_n)\ne(x_n)$ and $E(x_n)=E(y_n)$, which means that $E$ is not uniquely achievable.

If $E$ is not uniquely representable and its center of distances $\mathcal{C}(E)$ is not minimal, then $x_n=2x_{n+1}$, by the Banakiewicz Theorem. It remains to consider the case when $E$ is not uniquely representable and $\mathcal{C}(E)$ is minimal. Let $(y_n)$ be a non-fast convergent representation of $E$. In particular, $(x_n)\ne(y_n)$. Since $E=E(y_n)$ and
$$
\bigl\{y_n:\ n\in\mathbb N\,\bigr\}\ \subset \ \mathcal{C}(E)\ =\ \{0\}\,\cup\,\bigl\{x_n:\ n\in\mathbb N\,\bigr\},
$$
by Proposition \ref{p1}, we must have $\{x_n:\,n\in\mathbb N\,\}\setminus\{y_n:\,n\in\mathbb N\,\}\ne\emptyset.$ Thus, $(y_n)$ takes values only in $\{x_n:\,n\in\mathbb N\,\}$, and at least one value must be repeated in the sequence $(y_n)$. Moreover, not all values of $(x_n)$ occur as values of $(y_n)$, since $E(x_n)=E(y_n)$.

If $m$ is the smallest index among the $(x_n)$-terms omitted by the sequence $(y_n)$, that is, $m=\min\{k:\,\forall\,n\ x_k\ne y_n\,\}$, and if $l$ denotes the smallest index among the $(x_n)$-terms repeated in the sequence $(y_n)$, that is, $l=\min\{k:\,\exists\,n\ x_k=y_n=y_{n+1}\,\}$, then $m<l$, since $(x_n)$ is fast convergent and $\sum x_n=\sum y_n$.

Suppose, in addition, that no value $x_k$ is repeated in $(y_n)$ more than two times. Then
$$
\sum_{n=1}^\infty y_n\ =\ \sum_{n=1}^{m-1}y_n\ +\ \sum_{y_n\le x_{m+1}}y_n.
$$
Hence
$$
\sum_{n=1}^\infty y_n\ \le\ \sum_{n=1}^{m-1}x_n\,+\,\sum_{n=m+1}^\infty 2x_n\,=\,\sum_{n=1}^\infty x_n\,-\,x_m\,+\sum_{n=m+1}^\infty x_n \ <\ \sum_{n=1}^\infty x_n,
$$
by the fast convergence of $(x_n)$, a contradiction with $E(x_n)=E(y_n)$.

Thus, at least one value $x_p$ must appear at least three times in the sequence $(y_n)$, but then $2x_p\in\mathcal{C}(E)$, by Proposition \ref{p2}. Since $\mathcal{C}(E)$ is minimal, $2x_p=x_k$ for some $k<p$. Since $(x_n)$ is fast convergent, we must have $k=p-1$, that is, $x_{p-1}=2x_p$, which completes the proof.
\end{proof}

\section{The symmetric Cantor set range}

We now extend the above result to a larger class of symmetric Cantor sets, which have been studied on numerous occasions (see, for example, \cite{Feng97}, \cite{Kong24}, and \cite{Lu}). We begin with the standard geometric construction of a general symmetric Cantor set. Let $(n_k)$ be a sequence of integers with $n_k \ge 2$ and $(c_k)$ a sequence of positive numbers with $n_k c_k < 1$. Let $I_0$ be a closed interval $[a,b]$ and define $n_0:=1$. Suppose that $I_k$ consists of $n_k$ closed intervals $I_{k,i}$ of equal length. Define $I_{k+1}$ by removing $n_{k+1}-1$ open intervals of equal length from each $I_{k,i}$, leaving $n_{k+1}$ closed intervals, each of length $c_{k+1} |I_{k,i}|$. Then
$$
C_{(n_k),(c_k)} := \bigcap_{k=0}^\infty I_k = \bigcap_{k=0}^\infty \bigsqcup_{i=1}^{i_k} I_{k,i}
$$
is a \textsl{symmetric Cantor set}. Central Cantor sets form a subfamily of symmetric Cantor sets. Specifically, given a central Cantor set $C_{(a_n)}$ with ratios of dissection $(a_n) \in (0,1)^\mathbb{N}$, we have $C_{(a_n)} = C_{(2), (\frac{1 - a_k}{2})}$. A symmetric Cantor set $A$ with $\min A = 0$ corresponds to a semi-fast convergent sequence $(\alpha_k; N_k)$, i.e., $A = E(\alpha_k; N_k)$ \cite[pp. 1526-1527]{BFPW2}. For a symmetric Cantor set $C_{(n_k),(c_k)}$ with fundamental interval $[0,\eta]$, we have $N_k = n_k - 1$ and $\alpha_k = \eta \prod_{i=1}^k c_i$.

A value $\alpha_n$ of a semi-fast convergent sequence $(\alpha_n;N_n)$ is said to be irreducible if $\alpha_n\ne(N_{n+1}+1)\alpha_{n+1}$. Otherwise, we say that $\alpha_n$ is reducible. A semi-fast convergent sequence is called irreducible if all its values are irreducible
\cite[p. 20]{BBFP}. Every element of the center of distances of $E(\alpha_n;N_n)$ is a sum of finitely many terms of the semi-fast convergent sequence \cite[Lemmma 3.1]{BBFP}. Moreover, every semi-fast convergent sequence has infinitely many irreducible values.

It is not difficult to see that if $\alpha_i$ is reducible and $N_i=1$, then
$$
E(\alpha_n;N_n)\ =\ E(\alpha_1,\cdots, \alpha_{i-1}, \underbrace{\alpha_{i+1},\ldots,\alpha_{i+1}}_{2N_{i+1}+1}, \underbrace{\alpha_{i+2},\ldots,\alpha_{i+2}}_{N_{i+2}},\ldots).
$$
Similarly, if $\alpha_i$ is reducible and $N_i>1$, then
$$
E(\alpha_n;N_n)\ =\ E(\alpha_1,\cdots, \alpha_{i-1}, \underbrace{\alpha_i,\ldots,\alpha_i}_{N_i-1},\underbrace{\alpha_{i+1},\ldots,\alpha_{i+1}}_{2N_{i+1}+1}, \alpha_{i+2},\ldots).
$$
Thus, if $(\alpha_n;N_n)$ is not irreducible, then $E(\alpha_n;N_n)$ is not uniquely achievable.

\begin{theorem}
\label{t5}
Let $E=E(\alpha_n;M_n)$ be a symmetric Cantor set. $E$ is uniquely achievable if and only if $(\alpha_n;M_n)_{n\in\mathbb N}$ is irreducible and $M_n\le2$ for all $n$.
\end{theorem}

Unfortunately, there is no characterization of the minimality of centers of distances of achievement sets of semi-fast convergent sequences. Although some partial results on centers of distances of symmetric Cantor sets were presented in \cite{BBFP}, they are not sufficient to extend the method used in the proof of Theorem \ref{t4} to a proof of Theorem  \ref{t5}. Therefore, we need to undertake a detailed geometric analysis of symmetric Cantor sets and make use of some fine properties of achievement sets. The core of the proof of Theorem \ref{t5} is then a simple induction; however, the inductive step relies on the following lemma, which requires a longer proof.

\begin{lemma}
\label{l6}
Let $E$ be a symmetric Cantor set with the left endpoint at 0 and let $(\alpha_p;M_p)_{p\in\mathbb N}$ be an irreducible semi-fast convergent sequence such that $E=E(\alpha_p;M_p)$ and $M_p\le 2$ for all $p$. Let $(R_n, \alpha_n)$ be the longest dominating gap of $E$. Then, for any sequence $(x_i)_{i\in\mathbb N}$ with positive and monotone terms such that $E=E(x_i)$, it must be
\begin{equation}
\label{eq1}
x_i\ =\ y_i \quad\text{for all \ $i=1,2,\ldots,\,\sum_{p=1}^nM_p$}
\end{equation}
where $(y_i)_{i\in\mathbb N}=(\alpha_p;M_p)_
{p\in\mathbb N}$; that is, $y_i=\alpha_p$ for $\sum_{j=1}^{p-1}M_j<i\le \sum_{j=1}^pM_j$ and $p\in\mathbb N$.
\end{lemma}

\begin{proof}
By the definition of a semi-fast convergent sequence, we have $\alpha_n=y_m$ for $m=\sum_{j=1}^nM_j$. Since $(R_n,\alpha_n)$ is a dominating gap of $E$, it follows from the Third Gap Lemma that $y_m=x_k$ for some $k\in\mathbb N$. Moreover, $E_k(x_i)=E_m(y_i)=E_n(\alpha_p;M_p)$. Hence
$
\sum_{i>k}x_i=\sum_{i>m}y_i=\sum_{p>n}\alpha_pM_p,
$
and, consequently,
\begin{equation}
\label{eq2}
\sum_{i=1}^kx_i\ = \sum_{i=1}^my_i\ .
\end{equation}

It suffices to show that
\begin{equation}
\label{eq3}
x_{k-i}\ =\ y_{m-i}\quad \text{for } i\in\mathbb N_0,\quad 0\le i\le m-1.
\end{equation}
Indeed, \eqref{eq3} implies that
$$
\sum_{i=1}^mx_i
=\sum_{i=0}^{m-1}x_{k-i}
\overset{\eqref{eq3}}{=}
\sum_{i=0}^{m-1}y_{m-i}
=\sum_{i=1}^m y_i
\overset{\eqref{eq2}}{=}
\sum_{i=1}^kx_i.
$$
Hence $k=m$, and \eqref{eq1} follows.

If $k=1$, then $m=1$, $M_1=1$, and \eqref{eq1} follows directly from \eqref{eq2}. Thus, we may restrict our attention to the case $k>1$.

We prove \eqref{eq1} inductively. Let $l\in\mathbb N_0$, $l<m-1$, be such that
$
x_{k-i}=y_{m-i}
$
for all $i\in\mathbb N_0$, $0\le i\le l$. Since $l<m-1$, we have $m-l\ge2$. Furthermore,
$$
\sum_{i=k-l}^kx_i
=\sum_{i=m-l}^m y_i
\le \sum_{i=2}^m y_i
<\sum_{i=1}^m y_i
\overset{\eqref{eq2}}{=}
\sum_{i=1}^k x_i,
$$
and hence $k-l\ge2$ as well. Moreover, our inductive hypothesis gives
\begin{equation}
\label{eq31}
r_{k-i}^x=r_{m-i}^y
\qquad \text{for all }0\le i\le l+1.
\end{equation}

We will show that \eqref{eq3} implies
\begin{equation}
\label{eq4}
x_{k-l-1}\ =\ y_{m-l-1}.
\end{equation}

Let $p\in\mathbb N$ be such that $y_{m-l}=\alpha_p$. We divide the proof of \eqref{eq4} into two cases, according to the value of $M_p$.

\textsl{Case 1}. Suppose that
\begin{equation}
\label{eq5}
M_p=1.
\end{equation}
Since $y_{m-l-1}\ge y_{m-l}=\alpha_p$ and $M_p=1$, we have $p\ge2$ and
$
y_{m-l-1}=\alpha_{p-1}.
$

First, we eliminate the possibility that $x_{k-l-1}=x_{k-l}$. Suppose, to the contrary, that $x_{k-l-1}=x_{k-l}$. Then
$
2x_{k-l}=x_{k-l-1}+x_{k-l}\in E,
$
while
$
x_{k-l}+x_{k-l-1}>\alpha_p+R_p \overset{\eqref{eq5}}{=}R_{p-1}.
$
Consequently,
$$
2x_{k-l}\ \ge\ \min\bigg\{s\in E:\quad s>R_{p-1}\bigg\}\ =\ \alpha_{p-1}.
$$
Suppose first that $2x_{k-l}>\alpha_{p-1}$. Then
$
2x_{k-l}+R_p>\alpha_{p-1}+R_p.
$
Since the interval $(\alpha_{p-1}+R_p,\alpha_{p-1}+\alpha_p)$ is a gap of $E$, it follows that
$
2x_{k-l}+R_p\ge \alpha_{p-1}+\alpha_p.
$
Indeed,
$$
2x_{k-l}+R_p\ =\ x_{k-l}+x_{k-l-1}+r^y_{m-l}\ \overset{\eqref{eq31}}{=}\ x_{k-l}+x_{k-l-1}+r^x_{k-l}\in E.
$$
By the induction hypothesis, we have $x_{k-l}=y_{m-l}=\alpha_p$. Hence
$
\alpha_p+R_p=x_{k-l}+R_p\ge\alpha_{p-1},
$
which is a contradiction. Therefore, the equality $x_{k-l-1}=x_{k-l}$ necessarily implies
$
2x_{k-l}=\alpha_{p-1}.
$
That is,
$
(M_p+1)\alpha_p=2\alpha_p=\alpha_{p-1},
$
which contradicts the irreducibility of $(\alpha_i;M_i)_{i\in\mathbb N}$.

Thus, in the case \eqref{eq5}, we must have $x_{k-l-1}>x_{k-l}$. Consequently, either
\begin{equation}
\label{eq6}
x_{k-l-1}\,\in\,(\alpha_p,\alpha_p+R_p)
\end{equation}
or
\begin{equation}
\label{eq7}
x_{k-l-1}\ \ge\ \alpha_{p-1}.
\end{equation}

Ruling out the possibility in \eqref{eq6} requires a more detailed understanding of the structure of $E$. We have
$$
E\cap\bigl[0,y_{k-l-1}+r^y_{k-l}\bigr]\ =\ E\cap\bigl[0, \alpha_{p-1}+R_p\bigr]\ =\ E_{m-l}(y_i)\sqcup\bigl(y_{m-l}+E_{m-l}(y_i)\bigr)\sqcup\bigl(y_{m-l-1}+E_{m-l}(y_i)\bigr)
$$
and
$$
E_{m-l}(y_i)\ =\ E^m_{m-l+1}\,+\, E_m(y_i).
$$
The finite set $E^m_{m-l+1}(y_i)$ consists of $s$ elements, where
$
s=\prod_{i=m-l+1}^m(M_i+1).
$
We arrange these elements in increasing order:
$
e_1<\cdots<e_s.
$
In particular, $e_1=0$ and $e_2=\alpha_n$.

\begin{figure}[h]

\begin{tikzpicture}[
    >=latex, scale=0.85,
    font=\small,
    brace/.style={decorate,decoration={brace,amplitude=5pt}},
    seg/.style={densely dotted,very thick}
]

%--------------------------------------------------
% Parametry
%--------------------------------------------------
\def\L{0.75}      % długość przedziału
\def\G{0.42}      % zwykła przerwa
\def\Gs{0.18}     % krótka przerwa (3--4)

%--------------------------------------------------
% Makro rysujące jeden blok
% #1 = przesunięcie w osi x
% #2 = napis nad klamrą
% #3 = napis pod pierwszym początkiem (opcjonalnie)
%--------------------------------------------------

\newcommand{\DrawBlock}[3]{%

\begin{scope}[xshift=#1 cm]

\coordinate (A1) at (0,0);
\coordinate (B1) at (\L,0);

\coordinate (A2) at (\L+\G,0);
\coordinate (B2) at (2*\L+\G,0);

\coordinate (A3) at (2*\L+2*\G,0);
\coordinate (B3) at (3*\L+2*\G,0);

\coordinate (A4) at (3*\L+2*\G+\Gs,0);
\coordinate (B4) at (4*\L+2*\G+\Gs,0);

\coordinate (A5) at (4*\L+3*\G+\Gs,0);
\coordinate (B5) at (5*\L+3*\G+\Gs,0);

\coordinate (A6) at (5*\L+4*\G+\Gs,0);
\coordinate (B6) at (6*\L+4*\G+\Gs,0);

\foreach \A/\B in {
A1/B1,A2/B2,A3/B3,A4/B4,A5/B5,A6/B6}
{
    \draw[seg] (\A)--(\B);
}

% górna klamra
\def\End{6*\L+4*\G+\Gs}

\draw[brace]
(0,0.4) -- (\End,0.4)
node[midway,above=6pt] {$#2$};

% podpis pod początkiem pierwszego przedziału
\ifx&#3&
\else
\draw[->]
(A1)+(0,-0.75)--(A1)+(0,-0.05)
node[below=20pt,right] {$#3$};
\fi

\end{scope}
}

%--------------------------------------------------
% Pierwszy blok
%--------------------------------------------------

\DrawBlock
{0}
{E_{m-l}(y_i)}
{}

% oznaczenia nad

\node[above=1pt] at (0,0) {$e_1$};
\node[above=1pt] at (\L+\G,0) {$e_2$};
\node[above=1pt] at (5*\L+4*\G+\Gs,0) {$e_s$};

% oznaczenia pod

\node[left] at (0,0) {$0$};

\draw[->]
(\L+\G,-0.72)--(\L+\G,-0.05)
node[below=20pt,right]
{$e_2=\alpha_n=y_m=x_k$};

\node[below=1pt]
at (6*\L+4*\G+\Gs,0)
{$R_p$};

% dolna klamra pierwszego przedziału

\draw[brace,decoration={brace,mirror,amplitude=5pt}]
(0,-0.1)--(\L,-0.1)
node[midway,below=2pt]
{$E_m(y_i)$};

%--------------------------------------------------
% Drugi blok
%--------------------------------------------------

\DrawBlock
{6.7}
{y_{m-l}+E_{m-l}(y_i)}
{\alpha_p=y_{m-l}=x_{k-l}}

%--------------------------------------------------
% Trzeci blok
%--------------------------------------------------

\DrawBlock
{13.2}
{y_{m-l-1}+E_{m-l}(y_i)}
{\alpha_{p-1}=y_{m-l-1}}

\end{tikzpicture}
\caption{Pseudointervals} \label{fig1}
\end{figure}
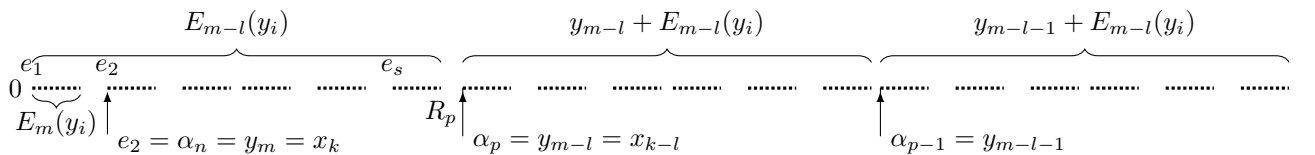

The dotted intervals on Figure \ref{fig1} represent translations of $E_m(y_i)$. Since they play a crucial role in the next part of the proof, it is convenient to introduce the following terminology: sets of the form $f+E_m(y_i)$ will be called \textsl{pseudointervals}. Figure \ref{fig1} depicts the case $s=6$ and $l=2$, which can occur, for example, when $M_n=2$, $M_{n+1}=1$, and $p=n+2$ (recall that we are considering the case \eqref{eq5}). Since the $E$-gap $(R_n,\alpha_n)$ is the longest dominating gap of $E$, all $E$-gaps contained in pseudointervals are shorter than $\alpha_n-R_n$. All other $E$-gaps of $E\cap\bigl[0,\alpha_{p-1}+R_p\bigr]$ (there are seventeen of them in the case depicted in Fig. 1) have length at most $\alpha_n-R_n$. The following observation provides another crucial ingredient for the next part of the proof.

\begin{obs}
Let $M:=\prod_{i=1}^nM_i$ and let $(L_j)_{j=1}^M$ be the sequence of all $E(y_i)$-gaps of order at most $n$, arranged in their natural order (that is, $\max L_j<\min L_{j+1}$). If the length of $L_j$ is smaller than $\alpha_n-R_n$, then necessarily $j<M$, and the length of $L_{j+1}$ is equal to $\alpha_n-R_n$.
\end{obs}

We are now ready to rule out the possibility that both \eqref{eq5} and \eqref{eq6} hold. By our inductive hypothesis,
\begin{equation}
\label{eq8}
E_{k-l}(x_i)\ =\ E_{m-l}(y_i)\ =\ \bigsqcup_{j=1}^s\bigl(e_j+E_m(y_i)\bigr).
\end{equation}
Moreover,
$$
E=E(x_i)\ \supset\ x_{k-l-1}+E_{k-l}(x_i)\ \overset{\eqref{eq8}}{=}\ \bigsqcup_{j=1}^s\bigl(x_{k-l-1}+e_j+E_m(y_i)\bigr).
$$
Denote the pseudointervals $x_{k-l-1}+e_j+E_m(y_i)$ by $\tilde{P}j$, for $j=1,\ldots,s$. Similarly, denote the basic pseudointervals $\alpha_p+e_j+E_m(y_i)$ by $P_j^{(1)}$, $j=1,\ldots,s$. Analogously, let
$P_j^{(2)}:=\alpha_{p-1}+e_j+E_m(y_i)$, $j=1,\ldots,s$, be the natural pseudointervals of $E\cap\bigl[\alpha_{p-1},\alpha_{p-1}+R_p\bigr]$, and let
$P_j^{(3)}:=\alpha_{p-1}+\alpha_p+e_j+E_m(y_i)$, $j=1,\ldots,s$, be the natural pseudointervals of $E\cap\bigl[\alpha_{p-1}+\alpha_p,\alpha_{p-1}+\alpha_p+R_p\bigr]$.

Suppose that \eqref{eq5} and \eqref{eq6} hold. We first observe that
$
x_{k-l-1}=\min\tilde{P}_1\ge\alpha_p+\alpha_n=y_{m-l}+y_m.
$
Indeed, otherwise we would have $\min\tilde{P}_1\in(\alpha_p,\alpha_p+R_p)$, since $(\alpha_p+R_p,\alpha_p+\alpha_n)$ is an $E$-gap and \eqref{eq6} holds. Clearly, $\max\tilde{P}_1\in E$. Since $(\alpha_p+R_n,\alpha_p+\alpha_n)$ is an $E$-gap, it follows that $\max\tilde{P}_1\in P_2^{(1)}$. Because $\tilde{P}_1\subset E$, there is a $\tilde{P}_1$-gap containing the $E$-gap $(\alpha_p+R_n,\alpha_p+\alpha_n)$. On the other hand, $\tilde{P}_1$ is a translation of $E_m(y_i)$, and hence $E_m(y_i)$ contains a gap of length at least $\alpha_n-R_n$, contradicting the domination of $(R_n,\alpha_n)$.

Thus, in the case of \eqref{eq5} and \eqref{eq6}, we have
$x_{k-l-1}\in[\alpha_p+\alpha_n,\alpha_p+R_p)\cap E$.
Let $d\in\{2,\ldots,s\}$ be such that
$x_{k-l-1}=\min\tilde{P}_1\in P_d^{(1)}$.
The following collection of equalities is the key observation:
\begin{equation}
\label{eq9}
\tilde{P}_i\ \ =\ \ \begin{cases} \ P_{d+i-1}^{(1)}\qquad &\text{for $i\in\{1,\ldots,s-d+1\}$;} \\
P_{i-s+d-1}^{(2)} &\text{ for $i\in\{s-d+2,\ldots,s\}$.}
\end{cases}
\end{equation}
We postpone the proof of \eqref{eq9} until later and first complete the proof that \eqref{eq5} and \eqref{eq6} cannot hold simultaneously.

Since $E_{k-l}(x_i)=E_{m-l}(y_i)$, the convex hulls of the translates $y_{m-l}+E_{m-l}(y_i)$ and $x_{k-l-1}+E_{k-l}(x_i)$ have the same length. That is,
$
\max P_s^{(1)}-\min P_1^{(1)}
=\max\tilde{P}_s-\min\tilde{P}_1
=\max P_{d-1}^{(2)}-\min P_d^{(1)}.
$
It follows that the sum of the lengths of all $s-1$ $E$-gaps between the pseudointervals $P_i^{(1)}$, $i=1,\ldots,s$, is equal to the sum of the lengths of all gaps between the pseudointervals $\tilde{P}_i$, $i=1,\ldots,s$. By \eqref{eq9}, the latter is precisely the sum of the lengths of all gaps between
$P_d^{(1)},\ldots,P_s^{(1)},P_1^{(2)},\ldots,P_{d-1}^{(2)}$.

If $d\ge 3$, the gaps $(\max P_i^{(2)},\min P_{i+1}^{(2)})$, $i=1,\ldots,d-2$, are translates of the gaps $\bigl(\max P_i^{(1)},\min P_{i+1}^{(1)}\bigr)$. It follows that the gaps $(R_{p-1},\alpha_{p-1})$ and $(\max P_{d-1}^{(1)}, \min P_d^{(1)})$ have equal length (see Figure \ref{fig2}).

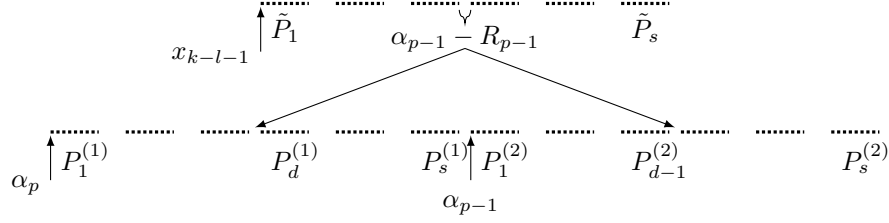
\begin{figure}[h]
\centering
\begin{tikzpicture}[
    >=latex,
    scale=0.85,
    font=\small,
    seg/.style={densely dotted,very thick},
    brace/.style={decorate,decoration={brace,amplitude=5pt}}
]

\def\Len{0.75}
\def\Gap{0.42}
\def\ShortGap{0.18}

% odległość między dwoma dolnymi blokami
\def\BlockShift{6.54}

% szerokość jednego bloku
\pgfmathsetmacro{\BlockWidth}{6*\Len+4*\Gap+\ShortGap}

% przesunięcie górnego bloku (środek nad całością)
\pgfmathsetmacro{\TopShift}{(\BlockShift+\BlockWidth-\BlockWidth)/2}

%%%%%%%%%%%%%%%%%%%%%%%%%%%%%%%%%%%%%%%%%%%%%%%%%%
% Makro pojedynczego bloku
%%%%%%%%%%%%%%%%%%%%%%%%%%%%%%%%%%%%%%%%%%%%%%%%%%

\newcommand{\DrawBlock}[1]{%
\begin{scope}[xshift=#1 cm]

\coordinate (A1) at (0,0);
\coordinate (B1) at (\Len,0);

\coordinate (A2) at ({\Len+\Gap},0);
\coordinate (B2) at ({2*\Len+\Gap},0);

\coordinate (A3) at ({2*\Len+2*\Gap},0);
\coordinate (B3) at ({3*\Len+2*\Gap},0);

\coordinate (A4) at ({3*\Len+2*\Gap+\ShortGap},0);
\coordinate (B4) at ({4*\Len+2*\Gap+\ShortGap},0);

\coordinate (A5) at ({4*\Len+3*\Gap+\ShortGap},0);
\coordinate (B5) at ({5*\Len+3*\Gap+\ShortGap},0);

\coordinate (A6) at ({5*\Len+4*\Gap+\ShortGap},0);
\coordinate (B6) at ({6*\Len+4*\Gap+\ShortGap},0);

\foreach \A/\B in {
A1/B1,
A2/B2,
A3/B3,
A4/B4,
A5/B5,
A6/B6}
{
    \draw[seg] (\A)--(\B);
}

\end{scope}
}

%%%%%%%%%%%%%%%%%%%%%%%%%%%%%%%%%%%%%%%%%%%%%%%%%%
% GÓRNY BLOK
%%%%%%%%%%%%%%%%%%%%%%%%%%%%%%%%%%%%%%%%%%%%%%%%%%

\begin{scope}[xshift=3.27cm,yshift=2cm]

\DrawBlock{0}

\end{scope}

%%%%%%%%%%%%%%%%%%%%%%%%%%%%%%%%%%%%%%%%%%%%%%%%%%
% DOLNE BLOKI
%%%%%%%%%%%%%%%%%%%%%%%%%%%%%%%%%%%%%%%%%%%%%%%%%%

\DrawBlock{0}
\DrawBlock{\BlockShift}

%%%%%%%%%%%%%%%%%%%%%%%%%%%%%%%%%%%%%%%%%%%%%%%%%%
% Oznaczenia
%%%%%%%%%%%%%%%%%%%%%%%%%%%%%%%%%%%%%%%%%%%%%%%%%%

\draw[->] (0,0)+(0,-0.75)--(0,-0.05);

\node[below=20pt,left] at (0,0)
{$\alpha_p$};

\node[below right] at (0,0)
{$P_1^{(1)}$};

\node[below right] at ({5*\Len+4*\Gap+\ShortGap},0)
{$P_s^{(1)}$};
\node[below right] at ({3*\Len+2*\Gap+\ShortGap},0)
{$P_d^{(1)}$};

\begin{scope}[xshift=\BlockShift cm]

\draw[->] (0,0)+(0,-0.75)--(0,-0.05);

\node[below=20pt] at (0,0)
{$\alpha_{p-1}$};

\node[below right] at (0,0)
{$P_1^{(2)}$};

\node[below right] at ({5*\Len+4*\Gap+\ShortGap},0)
{$P_s^{(2)}$};
\node[below right] at ({2*\Len+2*\Gap},0)
{$P_{d-1}^{(2)}$};
\end{scope}

\begin{scope}[xshift=3.27cm,yshift=2cm]

\draw[->] (0,0)+(0,-0.75)--(0,-0.05);

\node[below=20pt,left] at (0,0)
{$x_{k-l-1}$};

\node[below right] at (0,0)
{$\tilde{P}_1$};

\node[below right] at ({5*\Len+4*\Gap+\ShortGap},0)
{$\tilde{P}_s$};

\draw[brace,decoration={brace,mirror,amplitude=5pt}]
(3*\Len+2*\Gap,-0.1)--(3*\Len+2*\Gap+\ShortGap,-0.1)
node[midway,below=2pt]
{$\alpha_{p-1}-R_{p-1}$};

\end{scope}
\draw[->] (6*\Len+4*\Gap+3/2*\ShortGap,1.3)--(3*\Len+2*\Gap+\ShortGap/2,0.07);
\draw[->] (6*\Len+4*\Gap+3/2*\ShortGap,1.3)--(9*\Len+6*\Gap+5*\ShortGap/2,0.07);
\end{tikzpicture}

\caption{The gaps on the left of $P_d^{(1)}$ and on the right of $P_{d-1}^{(2)}$ must have the same length as the gap between $P_s^{(1)}$ and $P_1^{(2)}$ which is equal to $\alpha_{p-1} - R_{p-1}$.}
\label{fig2}

\end{figure}

Similarly, $E\supset x_{k-l-1}+x_{k-l}+E_k(x_i) \,=\,\bigsqcup_{j=1}^s\tilde{\tilde{P}}_j$, where the pseudointervals $\tilde{\tilde{P}}_j$ are given by $\tilde{\tilde{P}}_j:=x_{k-l-1}+x_{k-l}+e_j+E_m(y_i)$. By an argument fully analogous to the proof of \eqref{eq9}, we obtain
$$
\tilde{\tilde{P}}_i\ \ =\ \ \begin{cases} \ P_{d+i-1}^{(2)}\qquad &\text{for $i\in\{1,\ldots, s-d+1\}$;} \\
P_{i-s+d-1}^{(3)} &\text{ for $i\in\{s-d+2,\ldots,s\}$.}
\end{cases}
$$
It follows that the gaps $(\alpha_{p-1}+R_p,\alpha_{p-1}+\alpha_p)$ and $(\max P_{d-1}^{(2)}, \min P_d^{(2)})$ have equal length (analogously to Figure \ref{fig2}). Since $(\max P_{d-1}^{(2)}, \min P_d^{(2)})$ is a translation of $(\max P_{d-1}^{(1)}, \min P_d^{(1)})$, we conclude that the gaps $(R_{p-1},\alpha_{p-1})$ and $(R_p,\alpha_p)$ have the same length. Thus, $\alpha_{p-1}=2\alpha_p$, which contradicts the irreducibility of the representation $(\alpha_i;M_i)_{i\in\mathbb N}$, because $M_p=1$. Therefore, the conjunction of \eqref{eq5} and \eqref{eq6} is impossible.

To complete the justification of the previous statement, we now prove \eqref{eq9}. Since $\max P_d^{(1)}$ is a left endpoint of an $E$-gap and $\min \tilde{P}_1$ is a left endpoint of a pseudointerval contained in $E$, we must have $\min \tilde{P}_1\ne\max P_d^{(1)}$. Thus, either $\min\tilde{P}_1=\min P_d^{(1)}$ or $\min\tilde{P}_1\in(\min P_d^{(1)},\max P_d^{(1)})$.

The second possibility cannot occur. To see this, first consider the case $d=s$. Then $\max\tilde{P}_1\in\bigl[\min P_1^{(2)},\max P_1^{(2)}\bigr)$. It follows that either $\min\tilde{P}_2\in\bigl(\min P_1^{(2)},\max P_1^{(2)}\bigr]$ or $\min \tilde{P}_2\ge\min P_2^{(2)}$.

If $\min\tilde{P}_2\in\bigl(\min P_1^{(2)}, \max P_1^{(2)}\bigr]$, then $\max \tilde{P}_2\in P_2^{(2)}$, and hence the pseudointerval $\tilde{P}_2$ contains a gap of length at least $\min P_2^{(2)}- \max P_1^{(2)}=\alpha_n-R_n$. Thus, $E_m(y_i)$ contains a gap of the same length, which contradicts the domination of $(R_n,\alpha_n)$. Therefore, $\min \tilde{P}_2\ge\min P_2^{(2)}$. But then $\min\tilde{P}_2-\max\tilde{P}_1\,>\,\min P_2^{(2)}-\max P_1^{(2)}\,=\,\alpha_n-R_n$. Hence, a translate of $E_{m-l}(y_i)$ contains a gap longer than the longest gap of $E$, which is a contradiction.

Now, consider the remaining case $d<s$. Then $\min\tilde{P}_1\in\bigl(\min P_d^{(1)},\max P_d^{(1)}\bigr)$ implies $\max \tilde{P}_1\in\bigl(\min P_{d+1}^{(1)},\max P_{d+1}^{(1)}\bigr)$, which leads to a contradiction with the domination of $(R_n,\alpha_n)$ if $\min P_{d+1}^{(1)}-\max P_d^{(1)} =\alpha_n-R_n$. Otherwise, the gap between the pseudointervals $P_d^{(1)}$ and $P_{d+1}^{(1)}$ is shorter than $\alpha_n-R_n$. Then $d\le s-2$, because $\min P_s^{(1)}-\max P_{s-1}^{(1)}=\alpha_n-R_n$. Thus, by our Observation, the $\bigl(\alpha_p+E_{m-l}(y_i)\bigr)$-gap between the pseudointervals $P_{d+1}^{(1)}$ and $P_{d+2}^{(1)}$ has length exactly $\alpha_n-R_n$. This forces either $\tilde{P}_2$ to contain a gap of length at least $\alpha_n-R_n$ (a contradiction!) or the $(x_{k-l-1}+E_{m-l}(y_i))$-gap $\bigl(\max\tilde{P}_1, \min\tilde{P}_2\bigr)$ to be longer than $\alpha_n-R_n$. It follows that $E_{m-l}(y_i)$ contains an $E$-gap longer than $\alpha_n-R_n$, which is again a contradiction (see Figure \ref{fig3}).

Therefore, the first possibility, $\min\tilde{P}_1=\min P_d^{(1)}$, holds, and it yields $\tilde{P}_1=P_d^{(1)}$, because both are pseudointervals of the same length.

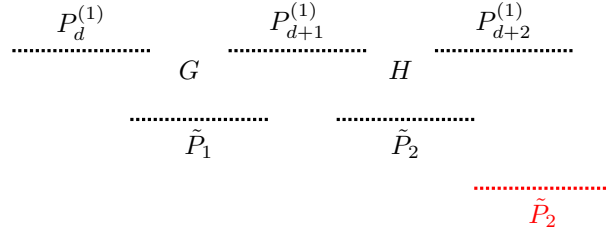
\begin{figure}[h]
\centering

\begin{tikzpicture}[
    scale=1.3,
    font=\small,
    seg/.style={densely dotted,very thick},
    redseg/.style={densely dotted,very thick,red}
]

\def\L{1.4}

%%%%%%%%%%%%%%%%%%%%%%%%%%%%%%%%%%%%
% GÓRNE PRZEDZIAŁY
%%%%%%%%%%%%%%%%%%%%%%%%%%%%%%%%%%%%

\draw[seg] (0,0)--(\L,0);
\node[above] at (0.7,0)
{$P_d^{(1)}$};

\draw[seg] (2.2,0)--(2.2+\L,0);
\node[above] at (2.9,0)
{$P_{d+1}^{(1)}$};

\draw[seg] (4.3,0)--(4.3+\L,0);
\node[above] at (5.0,0)
{$P_{d+2}^{(1)}$};

%%%%%%%%%%%%%%%%%%%%%%%%%%%%%%%%%%%%
% DRUGI POZIOM
%%%%%%%%%%%%%%%%%%%%%%%%%%%%%%%%%%%%

% \tilde P_1

\draw[seg]
(1.2,-0.7)--(1.2+\L,-0.7);

\node[below] at (1.9,-0.7)
{$\tilde P_1$};

% \tilde P_2

\draw[seg]
(3.3,-0.7)--(4.7,-0.7);

\node[below] at (4.0,-0.7)
{$\tilde P_2$};

%%%%%%%%%%%%%%%%%%%%%%%%%%%%%%%%%%%%
% TRZECI POZIOM
%%%%%%%%%%%%%%%%%%%%%%%%%%%%%%%%%%%%

\draw[redseg]
(4.7,-1.4)--(6.1,-1.4);

\node[below,red] at (5.4,-1.4)
{$\tilde P_2$};

\node[below] at (1.8,0.0)
{$G$};
\node[below] at (3.95,0.0)
{$H$};
\end{tikzpicture}

\caption{If the gap $G$ between $P_d^{(1)}$ and $P_{d+1}^{(1)}$ has length equal to $\alpha_n-R_n$, then also $\tilde{P}_1$ contains a gap of such a length which leads to a contradiction with the fact that $(R_n,\alpha_n)$ is dominating. If $G$ is shorter, than, by Observation, the gap $H$ between $P_{d+1}^{(1)}$ and $P_{d+2}^{(2)}$ has length equal to $\alpha_n-R_n$. If the left endpoint of $\tilde{P}_2$ lies in $P_{d+1}^(1)$ (black, upper variant), then the right endpoint must lie in $P_{d+2}^{(1)}$, and so $\tilde{P}_2$ contains a gap of length $\alpha_n-R_n$, a contradiction. If the left endpoint of $\tilde{P}_2$ lies in $P_{d+2}^(1)$ (red, lower variant), then the distance between $\tilde{P}_1$ and $\tilde{P}_2$ is longer than $\alpha_n-R_n$, a contradiction.}
\label{fig3}
\end{figure}

Suppose now that the equalities from \eqref{eq9} hold for $i=1,\ldots, t$ for some $t\le s-2$. Consider the family of pseudointervals
$$
\bigl\{W_i:\quad i=1,\ldots,s,\bigr\}\ :=\ \bigl\{P_{d+i-1}^{(1)}:\quad i=1,\ldots,s-d+1,\bigr\}\ \cup\ \bigl\{P_{i-s+d-1}^{(2)}:\quad i=s-d,\ldots,s,\bigr\}
$$
with the natural order inherited from the real line. In particular, $\tilde{P}_t=W_t$. If the $E$-gap between $W_t$ and $W_{t+1}$ has length $\alpha_n-R_n$, then $\min\tilde{P}_{t+1}$ must equal $\min W_{t+1}$; otherwise, the $(x_{k-l-1}+E_{k-l}(x_i))$-gap $(\max\tilde{P}_t, \min\tilde{P}_{t+1})$ would be longer than $\alpha_n-R_n$, and hence $E_{k-l}(x_i)$ would contain a gap of the same length, which is a contradiction. Thus, $\min\tilde{P}_{t+1}=\min W_{t+1}$, and consequently $\tilde{P}_{t+1}= W_{t+1}$, since both are pseudointervals.

If the $E$-gap between $W_t$ and $W_{t+1}$ is shorter than $\alpha_n-R_n$, then, by our Observation, the $E$-gap between $W_{t+1}$ and $W_{t+2}$ has length $\alpha_n-R_n$. In this case, $\min \tilde{P}_{t+1}$ cannot be greater than or equal to $\min W_{t+2}$, and hence $\min\tilde{P}_{t+1}\in W_{t+1}$. If $\min\tilde{P}_{t+1}\in(\min W_{t+1},\max W_{t+1})$, then $\max \tilde{P}_{t+1}\in W_{t+2}$. Hence, there is an $\bigl(x_{k-l-1}+E_{k-l}(x_i)\bigr)$-gap between $\min\tilde{P}_{t+1}$ and $\max\tilde{P}_{t+1}$ of length at least $\min W_{t+2}-\max W_{t+1}=\alpha_n -R_n$. It follows that $E_k(x_i)$ contains an $E$-gap of length at least $\alpha_n-R_n$, contradicting the domination of $(R_n,\alpha_n)$ (similarly to Figure \ref{fig3}). Therefore, we must have $\min\tilde{P}_{t+1}=\min W_{t+1}$, and it follows that $\tilde{P}_{t+1}=W_{t+1}$.

If the equalities from \eqref{eq9} hold for $i=1,\ldots, s-1$, then the gaps $(\max W_{s-1},\min W_s)$ and \linebreak $(\max\tilde{P}_{s-1},\min\tilde{P}_s)$ both have length $\alpha_n-R_n$, by the symmetry of $E_{m-l}(y_i)$ with respect to the point $\frac12r_{m-l}^y$. It follows that $\tilde{P}_s=\alpha_n-R_n+r^x_k+\tilde{P}_{s-1}=\alpha_n-R_n+r^y_{m}+W_{s-1}=W_s$. This completes the proof of \eqref{eq9}, which, in turn, completes the proof that \eqref{eq5} and \eqref{eq6} cannot hold simultaneously.

Therefore, in the case \eqref{eq5}, it must be that \eqref{eq7} holds. Then either $x_{k-l-1}>\alpha_{p-1}$ or $x_{k-l-1}=\alpha_{p-1}$. In the first case,
$$
r_{k-l-1}^x\ =\ \sum_{j=k-l}^\infty x_j\ =\ \sum_{j=m-l}^\infty y_j\ =\ r_{m-l-1}^y\ =\ R_{p-1}\ <\ \alpha_{p-1}\ <\ x_{k-l-1}.
$$
Thus, since $E=E(x_i)$, the interval $(R_{p-1},x_{k-l-1})$ is an $E$-gap, which contradicts $\alpha_{p-1}\in E$. Therefore, if \eqref{eq3} holds, we must have $x_{k-l-1}=\alpha_{p-1}=y_{m-l-1}$, and thus \eqref{eq4} holds in the case \eqref{eq3}.

\textsl{Case 2}. Suppose that
\begin{equation}
\label{eq10}
M_p\ =\ 2.
\end{equation}
Then either $y_{m-l+1}=y_{m-l}$ or $y_{m-l+1}<y_{m-l}$.

If $y_{m-l+1}=y_{m-l}=\alpha_p$, then $y_{m-l-1}=\alpha_{p-1}$, by \eqref{eq10}. We will show that, in this case, we cannot have $x_{k-l-1}>x_{k-l}$. If $x_{k-l-1}>2\alpha_p$, then $x_{k-l-1}>x_{k-l}+r_{k-l}^x$. Hence $\bigl(r_{k-l-1}^x, x_{k-l-1}\bigr)$ is an $E$-gap, while $2\alpha_p=y_{m-l+1}+y_{m-l}\in E$ lies in this gap, which is a contradiction.

If $x_{k-l-1}=2\alpha_p$, then
$$
E_{k-l-1}(x_i)\ =\ E\cap[0,\alpha_p+R_p]\ =\ \bigsqcup_{j=1}^{2s}V_j
$$
where the pseudointervals $V_j$ are defined by $V_j:=e_j+E_m(y_i)$ if $j\in{1,\ldots,s,}$ and by $V_j:=\alpha_p+e_{j-s}+E_m(y_i)$ if $j\in{s+1, \dots,s,}$. Putting $\tilde{V}_j:=x_{k-l-1}+V_j$, we obtain
$$
\bigsqcup_{j=1}^{2s}\tilde{V}_j\ =\ x_{k-l-1}+E_{k-l-1}(x_i)\ \subset \ E.
$$
Obviously, if $x_{k-l-1}=2\alpha_p$, then $\tilde{V}_j=2\alpha_p+V_j$ for $j=1,\ldots,s$. Moreover, by an argument analogous to the proof of
\eqref{eq9}, we see that $\tilde{V}_j=\alpha_{p-1}+V_{j-s}$ for $j=s+1,\ldots,s$.
Hence, the $E$-gaps $(R_p,\alpha_p)$ and $(2\alpha_p+R_p,\alpha_{p-1})$ have the same length (see Figure \ref{fig4}). It follows that $\alpha_{p-1}=3\alpha_p=(M_p+1)\alpha_p$, which contradicts the irreducibility of $(\alpha_i;M_i)_{i\in\mathbb N}$.

\begin{figure}[h]
%\centering

\begin{tikzpicture}[
    scale=0.55,
    font=\small,
    seg/.style={densely dotted,very thick},
    brace/.style={decorate,decoration={brace,amplitude=5pt}}
]

%------------------------------------------------
% Parametry
%------------------------------------------------

\def\Len{0.75}
\def\Gap{0.42}
\def\ShortGap{0.18}

\pgfmathsetmacro{\BlockWidth}{6*\Len+4*\Gap+\ShortGap}

% odległości między blokami
\pgfmathsetmacro{\Ggap}{0.9}
\pgfmathsetmacro{\Extra}{0.55}

%------------------------------------------------
% Makro bloku
%
% #1 - przesunięcie x
% #2 - podpis pierwszego przedziału
% #3 - podpis ostatniego przedziału
% #4 - podpis punktu początku drugiego przedziału
% #5 - przesunięcie y
% #6 - 0 podpisy pod, 1 podpisy nad
%------------------------------------------------

\newcommand{\DrawBlock}[6]{%

\begin{scope}[xshift=#1cm,yshift=#5cm]

% współrzędne przedziałów

\coordinate (A1) at (0,0);
\coordinate (B1) at (\Len,0);

\coordinate (A2) at (\Len+\Gap,0);
\coordinate (B2) at (2*\Len+\Gap,0);

\coordinate (A3) at (2*\Len+2*\Gap,0);
\coordinate (B3) at (3*\Len+2*\Gap,0);

\coordinate (A4) at (3*\Len+2*\Gap+\ShortGap,0);
\coordinate (B4) at (4*\Len+2*\Gap+\ShortGap,0);

\coordinate (A5) at (4*\Len+3*\Gap+\ShortGap,0);
\coordinate (B5) at (5*\Len+3*\Gap+\ShortGap,0);

\coordinate (A6) at (5*\Len+4*\Gap+\ShortGap,0);
\coordinate (B6) at (6*\Len+4*\Gap+\ShortGap,0);

% przedziały

\foreach \A/\B in
{A1/B1,A2/B2,A3/B3,A4/B4,A5/B5,A6/B6}
{
    \draw[seg] (\A)--(\B);
}

% podpisy przedziałów

\ifnum#6=0

\node[above=2pt] at
({0.5*\Len},0)
{$#2$};

\node[above=2pt] at
({5.5*\Len+4*\Gap+\ShortGap},0)
{$#3$};

\else

\node[above=2pt] at
({0.5*\Len},0)
{$#2$};

\node[above=2pt] at
({5.5*\Len+4*\Gap+\ShortGap},0)
{$#3$};

\fi

% strzałka punktu początku drugiego przedziału

%\if#1=0
%
%\else
%
%\draw[->]
%(A1)+(0,-0.75)
%--(A1)+(0,-0.05)
%node[midway,left=3pt]
%{$#4$};
%
%\fi

\end{scope}
}

%------------------------------------------------
% Pozycje bloków
%------------------------------------------------

\pgfmathsetmacro{\Btwo}{\BlockWidth+\Ggap}

\pgfmathsetmacro{\Bthree}{2*\BlockWidth+2*\Ggap}

\pgfmathsetmacro{\Bfour}{3*\BlockWidth+2*\Ggap+\Extra}

%------------------------------------------------
% DOLNY RZĄD
%------------------------------------------------

\DrawBlock
{0}
{V_1}
{V_s}
{\alpha_n}
{0}
{0}
\draw[->]
(A2)+(0,-0.95)
--(A2)+(0,-0.05)
node[below=14 pt]
{$\alpha_n$};
\draw[brace,decoration={brace,mirror,amplitude=3pt}]
(6*\Len+4*\Gap+\ShortGap,-0.1)--(\Btwo,-0.1)
node[midway,below=2pt]
{$G$};

\DrawBlock
{\Btwo}
{V_{s+1}}
{V_{2s}}
{x_{k-l}=y_{m-l}=\alpha_p}
{0}
{0}
\draw[->]
(A1)+(0,-0.95)
--(A1)+(0,-0.05)
node[below=18 pt,right]
{$\alpha_p=x_{k-l}=y_{m-l}$};

\DrawBlock
{\Bthree}
{ }
{ }
{2\alpha_p}
{0}
{0}
\draw[->]
(A1)+(0,-0.95)
--(A1)+(0,-0.05)
node[below=14 pt]
{$2\alpha_p$};

\DrawBlock
{\Bfour}
{ }
{ }
{\alpha_{p-1}}
{0}
{0}
\draw[->]
(A1)+(0,-0.95)
--(A1)+(0,-0.05)
node[below=14 pt]
{$\alpha_{p-1}$};

%------------------------------------------------
% GÓRNY RZĄD
%------------------------------------------------

\DrawBlock
{\Bthree}
{\tilde V_1}
{\tilde V_s}
{x_{k-l-1}}
{1.8}
{1}
\draw[->]
(A1)+(0,-0.55)
--(A1)+(0,-0.05)
node[below=5 pt]
{$x_{k-l-1}$};
\draw[brace,decoration={brace,mirror,amplitude=3pt}]
(\Bthree+6*\Len+4*\Gap+\ShortGap,1.7)--(\Bfour,1.7)
node[midway,below=4pt]
{$H$};
\draw[brace,decoration={brace,amplitude=3pt}]
(\Bthree+6*\Len+4*\Gap+\ShortGap,0.1)--(\Bfour,0.1);
\DrawBlock
{\Bfour}
{\tilde V_{s+1}}
{\tilde V_{2s}}
{}
{1.8}
{1}

\end{tikzpicture}

\caption{Gaps $G$ and $H$ must have the same length.}
\label{fig4}
\end{figure}
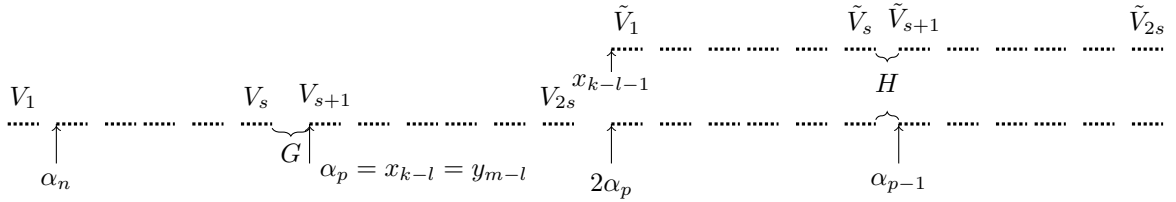

Thus, in the case $x_{k-l-1}>x_{k-l}$, we must have $x_{k-l-1}\in(\alpha_p,\alpha_p+R_p)$. Then, by considering the two translations $x_{k_l-1}+E_{k-l}(x_i)$ and $x_{k-l-1}+x_{k-l}+E_{k-l}(x_i)$, we arrive at a contradiction with the irreducibility of $(\alpha_i;M_i)_{i\in\mathbb N}$ in a manner analogous to the final part of the proof that \eqref{eq5} and \eqref{eq6} cannot hold simultaneously.

Therefore, in the case $y_{m-l+1}=y_{m-l}$, the only possible value of $x_{k-l-1}$ is $\alpha_p$, that is, $x_{k-l-1}=y_{m-l-1}$.

If $y_{m-l+1}<y_{m-l}=\alpha_p$, then $y_{m-l-1}=y_{m-l}$ by \eqref{eq10}. Clearly, $x_{k-l-1}\ge x_{k-l}$. If $x_{k-l-1}>x_{k-l}=y_{m-l}$, then $2\alpha_p=y_{m-l-1}+y_{m-l}\in(x_{k-l}+r_{k-l}^x, x_{k-l-1})$, and hence $2\alpha_p\not\in E(x_i)=E$, a contradiction. Thus, $x_{k-l-1}$ must be equal to $y_{m-l}$, and so to $y_{m-l-1}$. This completes the proof of Lemma \ref{l6}.
\end{proof}

\begin{proof}[Proof of Theorem \ref{t5}]
($\Rightarrow $) If $(\alpha_i;M_i)_{i\in\mathbb N}$ is not irreducible, then $E(\alpha_i;M_i)$ is not uniquely achievable, as we have seen at the beginning of this section. Now suppose that $M_{i_0}\ge 3$ for some $i_0\in\mathbb N$.

If $2\alpha_{i_0}=\alpha_j$ for some $j\in\mathbb N$, then $E(\alpha_i;M_i)=E(\alpha_i; N_i)$, where
$$
N_i\ \ =\ \ \begin{cases}
M_i+1 \qquad&\text{if $i=j$;}\\
M_{i_0}-2 &\text{if $i=i_0$;}\\
M_i &\text{otherwise.}
\end{cases}
$$

If $2\alpha_{i_0}\ne\alpha_j$ for all $j\in\mathbb N$, then $E(\alpha_i;M_i)=E(\beta_i; N_i)$, where $(\beta_i)$ is the unique monotone extension of $(\alpha_i)$ obtained by insertion of the value $2\alpha_{i_0}$ and where
$$
N_i\ \ =\ \ \begin{cases}
1 \qquad&\text{if $\beta_i=2\alpha_{i_0}$;}\\
M_i-2 &\text{if $\beta_i=\alpha_{i_0}$;}\\
M_i &\text{otherwise.}
\end{cases}
$$

($\Leftarrow$) Since $E$ has infinitely many gaps, it has infinitely many dominating gaps. Let $(y_i)_{i\in\mathbb N}=(\alpha_i;M_i)_{i\in\mathbb N}$, and let $(n_i)_{i\in\mathbb N}$ be an increasing sequence of indices such that $\bigl((R_{n_i},\alpha_{n_i})\bigr)_{i\in\mathbb N}$ is the sequence of all dominating $E$-gaps. Let $(x_i)$ be a sequence of positive, non-increasing terms such that $E=E(x_i)$. Applying Lemma \ref{l6} to $E(\alpha_i;M_i)$, we obtain that
$$
x_i\ =\ y_i \qquad\text{for}\ i=1,2,\ldots,\sum_{p=1}^{n_1}M_p.
$$

Observe that $(R_{n_{i+1}}, \alpha_{n_{i+1}})$ is the longest dominating gap in $E\cap[0,R_{n_i}]=E\bigl((\alpha_p;M_p){p>n_i}\bigr)$. Moreover,
$$
E\bigl((\alpha_p;M_p)_{p>n_i}\bigr)\ \ =\ \ E\bigl((x_p)_{p> \sum_{k=1}^{n_i}M_k}\bigr)
$$
provided that $x_j=y_j$ for $j=1,2,\ldots, \sum_{k=1}^{n_i}M_k$. Thus, again by Lemma \ref{l6}, we obtain $x_j=y_j$ for $j=1+\sum_{k=1}^{n_i}M_k, 2+\sum_{k=1}^{n_i}M_k, \sum_{k=1}^{n_{i+1}}M_k$. Therefore, by induction on $i$, we obtain $x_j=y_j$ for all $j\in\mathbb N$, which means that $E$ is uniquely achievable.
\end{proof}

The proof of Theorem \ref{t5} is significantly more complicated and longer than that of Theorem  \ref{t4}, because we did not have a characterization of the minimality of centers of distances for symmetric Cantor sets at our disposal. Finding such a characterization remains an interesting open problem.

\begin{problem}
Characterize the semi-fast convergent sequences that have minimal centers of distances.
\end{problem}

\section{The Cantorval range}

In the previous section, we examined when certain special achievable Cantor sets are uniquely determined. It is natural to consider other topological types of achievement sets. The situation is straightforward for finite unions of intervals, since they are never uniquely achievable (see \cite{BGM18}). The situation for Cantorvals, however, appears to be more difficult. The only known result in this direction is that the Guthrie-Nymann Cantorval $E(3,2;\frac14)$ has a unique representation (see \cite[Theorem 5.2]{BGM18}), and no other uniquely achievable Cantorvals are currently known.

The next theorem not only generalizes the unique achievability of the Guthrie-Nymann Cantorval, but also provides some new examples of uniquely achievable Cantorvals.

\begin{theorem}\label{Cantorvale unikalne}
Let $(x_n)$ be a summable decreasing sequence of positive terms such that $x_{2n}>r_{2n}$, $x_{2n-1}<r_{2n-1}$ and $x_{2n}-r_{2n} > x_{2n+2}-r_{2n+2}$ for every $n$. Then $E(x_n)$ is uniquely achievable.
\end{theorem}

\begin{proof}
Suppose that $E=E(y_n)$ for some nonincreasing sequence $(y_n)$ of positive terms. Let $r_n^x = \sum_{i>n}x_i$ and $r_n^y=\sum_{i>n}y_i$.

By our assumptions, every principal gap with respect to the sequence $(x_n)$ is dominating, and hence it is also principal with respect to the sequence $(y_n)$. Therefore, for every $n \in \N$ there is $k_{n}$ such that $y_{k_n} = x_{2n}$. Moreover, $r_{k_n}^y=r_{2n}^x$.

Therefore,
$$\sum_{i=k_n+1}^{k_{n+1}}y_i=r_{k_n}^y-r_{k_{n+1}}^y = r_{2n}^x-r_{2n+2}^x = x_{2n+1}+x_{2n+2} = x_{2n+1}+y_{k_{n+1}}.$$
Thus, $\sum_{i=k_n+1}^{k_{n+1}-1}y_i =x_{2n+1}$.

Suppose that $k_{n+1}-k_n > 2$. Then
$$
y_{k_{n+1}-1} \ \leq\ \tfrac12x_{2n+1}\ <\ \tfrac12r_{2n+1}^x \ =\ \tfrac12 x_{2n+2}+\tfrac12 r_{2n+2}^x\ <\ \tfrac12 x_{2n+2}+\tfrac12 x_{2n+2} \ =\ x_{2n+2} \,=\, y_{k_{n+1}},
$$
which contradicts the monotonicity of $(y_n)$.

Hence $k_{n+1}-k_n =2$, and so $y_{k_n+1}=x_{2n+1}$. Therefore, all terms of $(x_n)$ also appear in $(y_n)$. Since the sums of the two sequences are equal, $(x_n)=(y_n)$.
\end{proof}

In particular, every Cantorval of the form $E(k_1,k_2;q)$ satisfies the assumptions of Theorem \ref{Cantorvale unikalne}, and is therefore uniquely represented. For example, the Guthrie-Nymann Cantorval $E(3,2;\frac14)$ is such a set. However, there are other Cantorvals of this form. One such example is provided in \cite{NowCan}, where it is proved that the set $E(1,2\sqrt2-2;\frac{2-\sqrt2}{2})$ is a Cantorval. By Theorem \ref{Cantorvale unikalne}, it is uniquely achievable. Another example is $E(8,5;\frac14)$, which is a Cantorval by \cite[Theorem 8]{GK} (cf. \cite{Nit15}).

There are also two particular non-multigeometric sequences that deserve attention. One of them is defined by $b_n:=\frac1{2^n}+\frac{(-1)^n}{3^n}$ (see \cite[p.518]{Jones}). The second one is given by $a_{2n-1}=\frac1{2^{2n-1}}$ and $a_{2n}=\frac1{2^{2n}+1}$ (see \cite{PR25}). Thus, it is possible to apply Theorem \ref{Cantorvale unikalne} to both sequences and conclude that their achievement sets $E(a_n)$ and $E(b_n)$ are both uniquely determined. The paradox is that we do not know the topological type of either of these two achievement sets, and this problem appears to be quite challenging (cf. \cite[Problems 1.8 and 1.9]{GP26}).

Obviously, every Cantorval whose representation contains a term that is twice as large as another term  cannot have a unique representation. Similarly, a Cantorval cannot be uniquely achievable if a term repeats at least three times, or if a term repeats two times and another term is its triple.

The paper \cite{NowCan} contains a theorem that serves as a generator of infinitely many Cantorvals, mostly non-multigeometric ones. However, they all have geometry similar to that of the Guthrie-Nymann Cantorval. It is therefore natural to ask whether they share the property of unique achievability. Let us first recall the theorem from \cite{NowCan}.

\begin{theorem}\label{mami}

Let $(k_n)$, $(m_n)$ be sequences of natural numbers such that $k_1>1$ and $k_{n+1} > k_n+m_n$. Put $k_0:=0$, $m_0:=1,$ $K_n := {k_n,\dots,k_n+m_n-1}$, $K=\bigcup_{n=1}^\infty K_n$.
Let $x_1 > 0$ and for $i > 1$
$$
x_i \ \ := \ \ \begin{cases}
\frac12 x_{i-1}  &\text{ if } i-1,i,i+1 \notin K \text{ or } i-1,i \in K\\
\frac{2^{m_n}+1}{2^{m_n+1}} x_{i-1} &\text{ if } i=k_n-1> k_{n-1}+m_{n-1}\\
\frac{2^{m_n}+1}{2^{m_n+2}} x_{i-1} &\text{ if } i=k_n-1=k_{n-1}+m_{n-1}\\
\frac{2^{m_n}}{2^{m_n}+1} x_{i-1} &\text{ if }i=k_n \\
\frac14 x_{i-1} &\text{ if } i,i+1 \notin K, i-1\in K.
\end{cases}
$$
Then $E(x_n)$ is a Cantorval. Moreover, the sequence $(x_n)$ is such that $x_n > r_n$ if and only if $n \in K$.
\end{theorem}

\begin{proposition}
The Guthrie-Nymann Cantorval is the only Cantorval satisfying the assumptions of Theorem \ref{mami} which is uniquely achievable.
\end{proposition}

\begin{proof}
The Guthrie-Nymann Cantorval is the only Cantorval (up to multiplication by a constant) satisfying the assumptions of Theorem \ref{mami} with $K=\{2n\colon n\in\N\}$.

Let $(x_n)$ be a sequence satisfying the assumptions of Theorem \ref{mami} with $K\neq \{2n\colon n\in\N\}$. If there is $n\in\N$ such that $n,n+1,n+2 \notin K$ or $n,n+1 \in K$, then, by the definition, $x_{n+1}=\frac12x_n$, and so $E(x_n)$ is not uniquely achievable. Therefore, if $n \in K$, then $n+1 \notin K$, and since $K\neq \{2n\colon n\in\N\}$, there is $m \in \N$ such that $m,m+1\notin K$ and $m+2 \in K$. Then $x_{m+1} = \frac34 x_m$ and $x_{m+2} = \frac23x_{m+1} = \frac12x_m$. Thus, once again, there is a double of another term in $x_n$, so $E(x_n)$ is not uniquely achievable.
\end{proof}

There are other known families of Cantorvals that cannot be uniquely achievable. For example, it is not difficult to see that sequences generating Marchwicki-Miska Cantorvals or generalized Ferens Cantorvals (for definitions, see \cite{NP}, cf. \cite{MM,MNP}) must contain a duplication of another term. On the other hand, Kyiv Cantorvals (see \cite{NP}, cf. \cite{VMPS19}) have terms repeating at least three times.

Our observations on the unique achievability of Cantorvals are only preliminary. There is much to investigate in the future, and we conclude our note with two relevant and natural open problems.

\begin{problem}
Are there uniquely achievable Cantorvals other than those arising from Theorem \ref{Cantorvale unikalne}?
\end{problem}

\begin{problem}
Is it true that a Cantorval $E(x_n)$ is uniquely achievable if and only if the sequence $(x_n)$ does not contain neither a duplication of any term, nor a term repeating three times, nor a triple of a term that repeats two times?
\end{problem}

\end{document}